\documentclass[12 pt]{amsart}

\usepackage{hyperref}
\usepackage{etex}
\usepackage[shortlabels]{enumitem}
\usepackage{amsmath}
\usepackage{amsxtra}
\usepackage{amscd}
\usepackage{amsthm}
\usepackage{adjustbox}
\usepackage{amsfonts}
\usepackage{amssymb}
\usepackage{eucal}
\usepackage[all]{xy}
\usepackage{graphicx}
\usepackage{tikz-cd}
\usepackage{mathrsfs}
\usepackage{subfiles}
\usepackage{mathpazo}
\usepackage[colorinlistoftodos, textsize=tiny]{todonotes}
\usepackage{morefloats}
\usepackage{pdfpages}
\usepackage{thm-restate}
\usepackage[percent]{overpic}
\usepackage[utf8]{inputenc}
\usepackage{epigraph}
\usepackage{csquotes}
\usepackage[margin=1in]{geometry}
\usepackage{adjustbox}
\usepackage{microtype}
\usepackage{stmaryrd}
\usepackage{tikz}
\usetikzlibrary{shapes,arrows.meta}
\usepackage{tikz-3dplot} 
\usetikzlibrary{fadings}
\usetikzlibrary{decorations.pathmorphing, decorations.markings}

\usepackage{verbatim}
\usepackage{stmaryrd}
\usepackage{scalerel}
\usepackage{stackengine}
\stackMath
\newcommand\reallywidehat[1]{%
\savestack{\tmpbox}{\stretchto{%
  \scaleto{%
    \scalerel*[\widthof{\ensuremath{#1}}]{\kern-.6pt\bigwedge\kern-.6pt}%
    {\rule[-\textheight/2]{1ex}{\textheight}}
  }{\textheight}%
}{0.5ex}}%
\stackon[1pt]{#1}{\tmpbox}%
}
\usepackage{mathtools}

\graphicspath{ {images/} }

\RequirePackage{color}
\definecolor{myred}{rgb}{0.75,0,0}
\definecolor{mygreen}{rgb}{0,0.5,0}
\definecolor{myblue}{rgb}{0,0,0.65}

\usepackage{color}

\usepackage{hyperref}
\hypersetup{citecolor=blue}
\usepackage{tikz}
\usetikzlibrary{matrix,arrows,decorations.pathmorphing}

\theoremstyle{plain}
\newtheorem{theorem}[subsubsection]{Theorem}

\newtheorem{proposition}[subsubsection]{Proposition}
\newtheorem{proposition-definition}[subsubsection]{Proposition-Definition}
\newtheorem{lemma}[subsubsection]{Lemma}
\newtheorem{corollary}[subsubsection]{Corollary}

\theoremstyle{definition}
\newtheorem{definition}[subsubsection]{Definition}
\newtheorem{remark}[subsubsection]{Remark}

\newtheorem{question}[subsubsection]{Question}
\newtheorem{conjecture}[subsubsection]{Conjecture}

\newtheorem{statement}[subsubsection]{Statement}
\theoremstyle{remark}

\numberwithin{equation}{section}
\newcommand\nc{\newcommand}
\nc\on{\operatorname}
\nc\renc{\renewcommand}
\DeclareMathOperator\rk{rk}

\DeclareMathOperator\Mod{Mod}

\DeclareMathOperator\id{id}

\DeclareMathOperator\Jac{Jac}

\DeclareMathOperator{\Sp}{Sp}
\DeclareMathOperator{\Pic}{Pic}
\DeclareMathOperator{\ch}{ch}
\DeclareMathOperator{\td}{td}
\DeclareMathOperator{\VMHS}{VMHS}

\DeclareMathOperator{\Ext}{Ext}
\DeclareMathOperator{\Alb}{Alb}
\DeclareMathOperator{\Aut}{Aut}
\DeclareMathOperator{\Gal}{Gal}

\DeclareMathOperator{\Spec}{Spec}

\nc\mf\mathfrak
\nc\mc\mathcal
\nc\mb\mathbb
\nc\msf\mathsf
\nc\mscr\mathscr

\newcommand{\Z}{\mathbb{Z}}
\newcommand{\Q}{\mathbb{Q}}

\newcommand{\C}{\mathbb{C}}

\title{Section conjectures over $\C$ and Kodaira fibrations}
\author{Simon Shuofeng Xu\\ with an appendix by Seraphina Eun Bi Lee and Carlos A. Serv\'an}
\address{Department of Mathematics, University of Toronto}
\email{shuofeng.xu@mail.utoronto.ca}
\address{Department of Mathematics, Harvard University}
\email{slee@math.harvard.edu}
\address{Department of Mathematics, University of Chicago}
\email{cmarceloservan@uchicago.edu}

\date{\today}

\begin{document}

\begin{abstract}
In this paper we propose and study topological and Hodge theoretic analogues of Grothendieck's section conjecture over the complex numbers. We study these questions in the context of family of curves, in particular Kodaira fibrations, and in the context of the family of Jacobians associated to a Kodaira fibration. We showed that in the case of family of curves, both the topological and Hodge-theoretic analogues of the injectivity part of the section conjecture holds, and that the topological analogue of the surjectivity part of the section conjecture does not hold in general for families of curves (proven in the appendix written by Lee and Serv\'{a}n) and families of Jacobians.
\end{abstract}

\maketitle
\setcounter{tocdepth}{1}

\tableofcontents

\section{Introduction}\label{section:introduction}
In this paper, we would like to propose and study some analogues of Grothendieck's section conjecture over the complex numbers. To explain these analogues, we first recall the anabelian philosophy which motivates the present work. In \cite{Gro97}, Grothendieck conjectured that there exists a special class of schemes, called the anabelian schemes, defined over a field $k$ that is finitely generated over $\Q$, whose behaviour is controlled by an associated short exact sequence of \'{e}tale fundamental groups:
\begin{equation}
\label{etale SES}
1\to \pi_1^{\text{\'{e}t}}(X_{\overline{k}})\to \pi_1^{\text{\'{e}t}}(X)\to \Gal(\overline{k}/k)\to 1\end{equation}
Loosely speaking, this means that maps between anabelian schemes $X$ and $Y$ are in bijective correspondence with conjugacy classes of maps of extensions:
\[\begin{tikzcd}
	1 & {\pi_1^{\text{\'{e}t}}(X_{\overline{k}})} & {\pi_1^{\text{\'{e}t}}(X)} & {\Gal(\overline{k}/k)} & 1 \\
	1 & {\pi_1^{\text{\'{e}t}}(Y_{\overline{k}})} & {\pi_1^{\text{\'{e}t}}(Y)} & {\Gal(\overline{k}/k)} & 1
	\arrow[from=1-1, to=1-2]
	\arrow[from=1-2, to=1-3]
	\arrow["{\bar{f}}", from=1-2, to=2-2]
	\arrow[from=1-3, to=1-4]
	\arrow["f", from=1-3, to=2-3]
	\arrow[from=1-4, to=1-5]
	\arrow["{=}", from=1-4, to=2-4]
	\arrow[from=2-1, to=2-2]
	\arrow[from=2-2, to=2-3]
	\arrow[from=2-3, to=2-4]
	\arrow[from=2-4, to=2-5]
\end{tikzcd}\]
where two such maps $(f,\bar{f})$ and $(g,\bar{g})$ are conjugate if their images are conjugate by an element of $\pi_1^{\text{\'{e}t}}(Y_{\overline{k}})$.
\begin{remark}
	Recall that to define the \'{e}tale fundamental group of $X$, one needs a choice a geometric point of $X$. The equivalence relation on the set of morphisms of extensions is introduced precisely to account for this choice. This also justifies the omission of base points in our notation.
\end{remark}
Furthermore, Grothendieck conjectured that the class of anabelian schemes should satisfy the following properties (see \cite{Fal98}):
\begin{enumerate}
	\item it should contain all hyperbolic curves
	\item it should contain the moduli stacks $\mc{M}_{g,n}$ of smooth projective curves of genus $g$ with $n$ marked points 
	\item it should be closed under taking fibrations, i.e., if $f:X\to Y$ is a smooth proper map such that both $Y$ and all the fibers $X_y$ are also in this class, then so is $X$.
\end{enumerate}
Now if one also believes that a point $\Spec k$ is anabelian, then one arrives at the Grothendieck's section conjecture for smooth proper curves of genus $g\geq 2$
\begin{conjecture}[Grothendieck's section conjecture]
\label{Conj: Grothendieck's section conjecture}
	Let $X/k$ be a smooth projective curve of genus $g\geq 2$ over a field $k$ that is finitely generated over $\Q$. Then the section map $$\text{sec}:\{k\text{-rational points of }X\}\to \{\text{splittings of }(\ref{etale SES})\}/\text{conjugation}$$ is a bijection.
\end{conjecture}
\begin{remark}
	One can see that the notion of anabelian schemes is supposed to be an algebraic analogue of $k(\pi,1)$-spaces in topology. Indeed, hyperbolic curves and $\mc{M}_{g,n}$ are all $k(\pi,1)$-spaces, and if we have a Serre fibration $f:X\to Y$, where both $Y$ and the fiber $F$ are $k(\pi,1)$-spaces, then so is $X$. However, it's worth pointing out that the analogy is not perfect and there are varieties which are $k(\pi,1)$ spaces that are not anabelian (for example, elliptic curves over a field of characteristics $0$ are not anabelian but they are still $k(\pi,1)$ spaces).
\end{remark}

The question we would like to explore in this paper is that if we now work over the complex numbers, is there a reasonable class of schemes, and some functorial invariant $F$ like the \'{e}tale fundamental group, so that the functor which sends such a scheme $X$ to this functorial invariant $F(X)$ is fully faithful. A na\"{i}ve approach to this question is that we may simply take the class of anabelian schemes proposed by Grothendieck and see if one can find any interesting functorial invariant that can replace the \'{e}tale fundamental group when we work over the complex numbers.

In this paper, we study two invariants that naturally occur in complex algebraic geometry: the first one is the topological fundamental group and the other one is the category of graded-polarizable admissible $\Z$-variation of mixed Hodge structures. Furthermore, we study them in the context of families of curves over another curve, in particular Kodaira fibrations, and the family of Jacobians associated to a Kodaira fibration. We first recall what a Kodaira fibration is.
\begin{definition}[Kodaira fibration]
	A \textit{Kodaira fibration} $f:X\to B$ is a non-isotrivial fibration from a smooth projective surface $X$ onto a smooth projective curve $B$ such that all of the fibers are smooth projective of genus $g$.
\end{definition}

Examples of such a fibration are first constructed by Kodaira in \cite{Kod67}; see also the work of Parshin \cite{Par68} and Atiyah \cite{Ati69}. For a recent survey on Kodaira fibration, see \cite{Cat17}. By an theorem in a paper of Kas which he attributed to Kodaira and Bers \cite[Theorem 1.1] {Kas68}, we know that given a Kodaira fibration $f:X\to B$, the genus of the base $B$ is at least $2$ and the genus of the fiber $X_b$ is at least $3$. In particular, they should be examples of anabelian schemes, making them ideal testing grounds for anabelian conjectures. We can now formulate the questions we study in this paper:

\textbf{Topological version of the question}. Given a Kodaira fibration $f:X\to B$, we have a short exact sequence of topological fundamental groups
\begin{equation}
	\label{top. SES}
	1\to \pi_1(X_b)\to \pi_1(X)\to \pi_1(B)\to 1.
\end{equation}
Now any algebraic section induces a splitting of this short exact sequence up to conjugation by $\pi_1(X_b)$, so we get a section map $$\sec_{\text{top}}:\{\text{algebraic sections of }f:X\to B\}\to \{\text{sections of (\ref{top. SES})}\}/\text{conjugation},$$ where the conjugation action is the natural action of $\pi_1(S_b)$ on the set of sections of (\ref{top. SES}) defined by $$(g\cdot f)(x):=gf(x)g^{-1}\quad\quad\text{ for all }g\in \pi_1(X_b), x\in \pi_1(B), f\text{ splittings of (\ref{top. SES})}.$$ The following question is then a topological analogue of Grothendieck's section conjecture (Conj. \ref{Conj: Grothendieck's section conjecture}).
\begin{question}
\label{question:top. section question}
	Is the section map $\sec_{\text{top}}$ a bijection?
\end{question}

Now it's worth pausing and asking why one should think that Question \ref{question:top. section question} may have a positive answer. Indeed, if we just consider the na\"{i}ve topological analogue of Grothendieck's section conjecture for smooth projective curves $X$ over $\C$, we will not get an one-to-one correspondence: there are infinitely many maps from $\Spec \C$ to $X$, but only one group theoretic map from $\{*\}=\pi_1(\Spec \C)\to \pi_1(X)$. Therefore, the section map in this case is surjective but never injective. \footnote{One may repair this issue by working with $\pi_0(X(\C))$, and this approach is used in the formulation of the real section conjecture. See \cite[Section 16.1]{Sti13} for an overview.}

However, notice that in this case, $\pi_1(\Spec \C)$ is trivial, and therefore $X\to \Spec\C$ has trivial monodromy. On the other hand, in the case of a curve over a number field $k$, since $\pi_1^{\text{\'{e}t}}(k)=\Gal(\overline{k}/k)$, we do have non-trivial monodromy. Therefore, in some sense, Kodaira fibrations are closer analogues to curves over number fields than Riemann surfaces. Indeed, the analogue of Faltings' theorem, known as the geometric Mordell conjecture holds for Kodaira fibrations \cite{Man63}, i.e. any Kodaira fibration $f:X\to B$ has only finitely many algebraic sections. Furthermore, one can give a proof of geometric Mordell conjecture using a strategy similar to Faltings' original proof \cite{McM00}. In particular, it proves a geometric Shafarevich conjecture \cite[Theorem 3.1]{McM00}, which states that for a given base $B$, there are only finitely many Kodaira fibrations $f:X\to B$ with fiber genus $g(S_b)=g$.

Perhaps the most relevant evidence is that the topological analogue of Grothendieck's section conjecture holds for the universal curve $\mc{C}_g$ over the moduli stack $\mc{M}_g$ of genus $g$ curves when $g\geq 2$: for $\mc{C}_g\to \mc{M}_g$, the associated short exact sequence of topological fundamental group is the Birman short exact sequence $$1\to \pi_1(\Sigma_g)\to \text{MCG}_{g,1}\to \text{MCG}_g\to 1,$$ where $\Sigma_g$ is a smooth compact orientable Riemann surface of genus $g$, and $\text{MCG}_{g,n}$ is the mapping class group of a genus $g$ surface with $n$ marked points. This short exact sequence is known to be a non-split sequence whenever $g\geq 2$ \cite[Corollary 5.11]{FM12} (as we will use later, it does not even virtually split, see \cite{CS21}). Analogous results hold for the universal $n$-pointed curve as well \cite[Corollary 1.2]{Che19}. Therefore, it's interesting to ask if Question \ref{question:top. section question} has a positive answer at least when the monodromy representation $$\rho:\pi_1(B)\to \Sp_{2g}(\Z)$$ associated to a Kodaira fibration $f:X\to B$ has large image. One of the main results of this paper is that the non-injectivity phenomenon observed in the case of Riemann surfaces does not occur for Kodaira fibrations with large monodromy:
\begin{theorem}[See Cor. \ref{injectivity of top. section question} in Section \ref{section:family of Jacobians}]
If $f:X\to B$ is a Kodaira fibration whose monodromy representation $\rho$ has no invariants, then $\sec_{\text{top}}$ is injective.
\end{theorem}
The main strategy to prove this theorem is to study the associated family of Jacobians $\pi:\Pic^0_{X/B}\to B$ attached to a Kodaira fibration $f:X\to B$. One could similarly define an abelianized section map in this context and ask Question \ref{question:top. section question} in the case of family of Jacobians. Then the theorem above follows from the following one:
\begin{theorem}
	Let $\pi:\Pic^0_{X/B}\to B$ be the family of Jacobians associated to a Kodaira fibration. Then if the monodromy representation associated to $f:X\to B$ has no invariants, the abelianized section map is injective (see Cor. \ref{Cor: injectivity of phi^ab}).
\end{theorem}

In fact, this theorem is true for any abelian schemes over a curve, and we work in that generality at the beginning of section \ref{section:family of Jacobians}. 

On the other hand, it turns out that the surjectivity part of the topological section question is not true in general. 
\begin{proposition}\label{Prop: surjectivity of top. section conjecture}\hfill
\begin{enumerate}
	\item Let $f:X\to B$ be a Kodaira fibration with an algebraic section and $\pi:\Pic^0_{X/B}\to B$ is the associated family of Jacobians. Then the abelianized section map is never surjective (see Corollary \ref{surjectivity});
	\item Let $f:X\to B$ be a Kodaira fibration with an algebraic section. Then there exists a double cover $B'\to B$ such that the base change $f':X':=X\times_B B'\to B$ is a Kodaira fibration for which the topological section map is not surjective (see Appendix \ref{section: appendix A}).
\end{enumerate}
\end{proposition}
The second part of the proposition is due to Seraphina Lee and Carlos Serv\'{a}n who adapted their construction in \cite{LS24} of holomorphic Lefschetz fibrations with infinitely many sections to this setting. We are very grateful to them for writing up a proof of the second part of this proposition as an appendix to this paper.

We also show that Lee and Serv\'{a}n's construction in fact implies the existence of counterexamples to the surjectivity of the topological section map which have no algebraic sections. To explain this, we first introduce the weak topological section conjecture:
\begin{conjecture}[Weak topological section conjecture]
\label{conjecture: weak top. section conjecture}
	Let $f:X\to B$ be a Kodaira fibration. Then it admits an algebraic section if and only if the short exact sequence (\ref{top. SES}) of fundamental groups splits.
\end{conjecture}
Next we relate the surjectivity of topological section maps for Kodaira fibrations to this weak topological section conjecture. More precisely, we prove the following:
\begin{proposition}[see Cor. \ref{Cor: weak=surjectivity for Kodaira fibrations}]
	The surjectivity of the section map $\sec_{\text{top}}$ is equivalent to the weak topological section conjecture for all connected finite \'{e}tale covers $X'$ of $X$ such that $X'\to B$ has connected fibers (i.e. $X'\to B$ is also a Kodaira fibration).
\end{proposition}
\begin{remark}
	This result also has an analogue in the arithmetic setting \cite[Theorem 31]{Sti10}, which states that the surjectivity part of Grothendieck's section conjecture may be deduced from the weak section conjecture (i.e. existence of a rational points is the same as the existence of a splitting of the sequence \ref{etale SES}) for geometrically connected finite \'{e}tale covers of the curve. In fact, the proof is almost the same as the proof of this analogous result in \cite{Sti10}.
\end{remark}

Then combining these two results, we see that
\begin{corollary}
	There exists a Kodaira fibration with no algebraic section but whose associated short exact sequence of topological fundamental group splits.
\end{corollary}

\textbf{Hodge theoretic version of the question}. Another possible candidate to replace the \'{e}tale fundamental group is the category of graded-polarizable admissible $\Z$-variation of mixed Hodge structures $\VMHS_\Z$, as Hodge theory has always been an extremely useful tool in the study of complex algebraic geometry. The downside of this category is that it's not a Tannakian category, and therefore, we cannot use it to define a Tannakian fundamental group. This means that we will lose the group theoretic aspect of anabelian geometry if we work with $\VMHS_\Z$. Nevertheless, it still makes sense to ask if functors between these categories are in one-to-one correspondence with algebraic sections. In other words, given a map $f:X\to B$, we get a section map $$\text{sec}:\{\text{algebraic sections to }f\}\to \{\text{section functors from }\VMHS_\Z(X)\to \VMHS_\Z(B)\},$$ where section functors are defined to be functors that becomes isomorphic to the identity functor after composing with the pullback $f^*$. Then the Hodge theoretic section question asks if this Hodge theoretic section map is a bijection. See section (\ref{section: Hodge theoretic version}) for more detailed backgrounds and a more precise formulation of the question. 

For this Hodge theoretic analogue, we prove the following injectivity result:
\begin{theorem}[see Prop. \ref{Proposition: Hodge theoretic injectivity}]
	For any family of curves $f:X\to B$. The injectivity part of the Hodge theoretic section question holds, i.e. distinct algebraic sections produce non-isomorphic section functors.
\end{theorem}
It's worth pointing out that this result does not make any assumption on the monodromy of the family. Some special cases of this theorem can be readily deduced from known results. When $B$ is a point, this theorem follows from the Hain and Pulte's pointed Torelli theorem (see \cite[Theorem 5.5]{Pul88} and \cite[Theorem 7.5]{Hai87}). Hain, through private communication, also informed us that his proof of the main result in \cite{Hai11} can be adapted to prove that this Hodge theoretic section map is a bijection in the case where $B=\mc{M}_{g}$ for any $g>4$.

\textbf{Acknowledgement:} I would like to first thank my advisor Daniel Litt. This paper would not have been possible without his encouragement, generosity and wisdom. I'm grateful to Seraphina Lee and Carlos Serv\'{a}n for adapting their work into an appendix to this paper and for many helpful conversations. I would also like to thank Peter Jossen, Richard Hain, and Nick Salter for reading an earlier version of this paper and for giving me many valuable feedbacks and suggestions. Richard Hain and Sasha Shmakov point me to many useful references, for which I'm very grateful. Finally, I would like to thank Laure Flapan, whose delightful talk on Kodaira fibration got me interested in them in the first place.

\section{Some construction of Kodaira fibrations}\label{section:construction}
In this paper, we will prove theorems about Kodaira fibrations satisfying some extra hypotheses. In this section we would like to demonstrate that those theorems are non-empty, i.e., we give some standard constructions of Kodaira fibrations, which produce Kodaira fibrations satisifying those extra hypotheses. We also take this opportunity to set up some notations.

Let $f\colon X\to B$ be a Kodaira fibration and let $\mc{M}_g$ be the moduli stack of smooth projective curves of genus $g$. By the universal property of $\mc{M}_g$, such a Kodaira fibration $f\colon X\to B$ corresponds to some non-constant map $\varphi\colon B\to \mc{M}_g$. Therefore, to construct a Kodaira fibration, it is enough to construct complete curves inside $\mc{M}_g$. To avoid issues with stacks, we will work instead with the fine moduli space $\mc{M}_g[n]$ of genus $g$ curves with fixed level $n\geq 3$ structure. The following  construction, which we called the moduli construction, is fairly well-known (see for example \cite[Prop. 2.1]{Fla22}).

\textbf{Moduli construction}: Suppose $g\geq 4$ and consider the Satake compactification $\mc{M}_g[n]^*$, i.e., the closure of $\mc{M}_g[n]$ inside the Satake compactification $\mc{A}_g[n]^*$ of $\mc{A}_g[n]$ via the Torelli map $J\colon \mc{M}_g[n]\to \mc{A}_g[n]$. Since $\mc{M}_g[n]^*$ is projective, we may embed it into some large projective space and cut it with hyperplane sections and produce a curve. Now that when $g\geq 4$, the boundary component $\mc{M}_g[n]^*-J(\mc{M}_g[n])$ is of codimension at least $2$ and the hyperelliptic locus $\mc{H}_g[n]$, where the Torelli map fails to be an immersion, is of codimension at least $2$. Hence, a curve obtained by cutting $\mc{M}_g[n]^*$ with hyperplane sections can avoid the boundary component as well as the hyperelliptic locus, and corresponds to a Kodaira fibration via the universal property of $\mc{M}_g[n]$.

\begin{lemma}\label{lemma: surjective monodromy}
	Let $f\colon X\to B$ be a Kodaira fibration constructed via the moduli construction explained above. Then the image of the monodromy representation $$\rho\colon \pi_1(B)\to \Sp_{2g}(\mb{Z})$$ is of finite index inside $\Sp_{2g}(\Z)$ and hence the monodromy action on $H^1(X_b,\Z)$ has no invariants (i.e. $H^0(C,R^1\pi_*\Z)=0$).
\end{lemma}
\begin{proof}
	By Lefschetz's hyerplane theorem for quasi-projective varieties \cite[page 153]{GM88}, the monodromy representation $\rho$ factors through $\pi_1(\mc{M}_g[n])$ via a surjection: $$\pi_1(B)\twoheadrightarrow \pi_1(\mc{M}_g[n])\to \Sp_{2g}(\mb{Z}),$$ and by definition, the last map, which corresponds to the monodromy representation for the universal family over $\mc{M}_g[n]$, surjects onto the kernel of $\Sp_{2g}(\Z)\to \Sp_{2g}(\Z/n\Z)$, which certainly acts on $H^1(X_b,\Z)$ with no invariants.
\end{proof}

\begin{remark}\label{Remark: monodromy of Kodaira fibration}
	For any Kodaira fibration $f:X\to B$, a standard argument using Torelli's theorem (see \cite[page 2]{Fla22}) shows that the monodromy representation $\rho:\pi_1(B)\to \Sp_{2g}(\Z)$ always has infinite image.
\end{remark}
Since we are interested in studying algebraic sections of Kodaira fibrations, it would be quite pointless if there are no Kodaira fibrations with algebraic sections. The following theorem of Bregman proves that that's not the case, and that monodromy cannot obstruct the existence of an algebraic section.
\begin{proposition}\cite[Prop. 4.2]{Bre21}\label{Bregman's proposition}
	For every Kodaira fibration $f\colon X\to B$, there exists a Kodaira fibration $\tilde{f}\colon \tilde{X}\to \tilde{B}$ with an algebraic section such that 
	\begin{enumerate}
		\item The fibers of $f\colon X\to B$ and $\tilde{f}\colon \tilde{X}\to \tilde{B}$ are of same genus $g$ 
		\item The monodromy homomorphisms $\pi_1(B)\to \Mod_g$ and $\pi_1(\tilde{B})\to \Mod_g$ have the same image.
	\end{enumerate}
\end{proposition}
\begin{proof}[Sketch of a proof]
	The idea is similar to that of the moduli construction. First embed $X$ into some projective space $\mb{P}^N$. Then a general hyperplane section $\tilde{B}$ will be smooth and projective and admits a non-constant and hence finite map onto $C$. Then $\tilde{X}$ can be constructed as the fiber product of $\tilde{B}$ and $X$ over $B$ and the algebraic section comes from the identity map $\tilde{B}\to \tilde{B}$ and the natural inclusion map $\tilde{B}\to S$. By the Lefschetz hyperplane theorem, we know that $\pi_1(\tilde{B})$ surjects onto $\pi_1(X)$ and hence onto $\pi_1(B)$ and the monodromy representation factors through this surjection. This proves the statement on the monodromy homomorphism.
\end{proof}


\section{Algebraic sections to family of Jacobians}\label{section:family of Jacobians}

\subsection{Family of Jacobians and injectivity}
Results in this section holds more generally for any abelian schemes over a curve, so we first work in that generality. Let $p\colon \mc{A}\to B$ be an abelian scheme over a smooth projective curve $B$ and let $q\colon \mc{A}^\vee=\Pic^0_{\mc{A}/B}\to B$ be the associated dual abelian scheme. Associated to the map $p\colon \mc{A}\to B$ is a short exact sequence of topological fundamental groups $$1\to H_1(\mc{A}_b,\Z)\to \pi_1(\mc{A})\to \pi_1(B)\to 1.$$ Given an algebraic section $s\colon B\to \mc{A}$, we get a group theoretic splitting of this short exact sequence. It is well known that isomorphism classes of group theoretic splittings are parametrized by the cohomology group $H^1(\pi_1(B), H_1(\mc{A}_b,\Z))$ \cite[Chapter 8, Theorem 1.3]{Lan96}, and so the map $\sec_{\text{ab}}$ may be rewritten as $$\sec_{\text{ab}}\colon H^0(B,\mc{A})\to H^1(\pi_1(B), H_1(\mc{A}_b,\Z)),$$ where we view $\mc{A}$ as a sheaf over $B$ and $H^0(B,\mc{A})$ is the group of algebraic sections of $p\colon \mc{A}\to B$. We would like to relate $\sec_{\text{ab}}$ with a boundary map in some long exact sequence of cohomology groups coming from $q\colon \mc{A}^\vee\to B$ which we now explain. 

Consider the short exact sequence $0\to R^1q_*\Z(1)\to R^1q_*\mc{O}\to (\mc{A}^\vee)^\vee=\mc{A}\to 0$ induced by the exponential exact sequence. Over each $b\in B$, this short exact sequence is simply the universal covering map $H^1(\mc{A}_b^\vee,\mc{O})\to \mc{A}_b$ of the fiber $\mc{A}_b$ with kernel $H_1(\mc{A}_b^\vee,\Z)$. Taking the long exact sequence in cohomology, we get $$\dots \to H^0(B, R^1q_*\mc{O})\to H^0(B, \mc{A})\xrightarrow{\psi}H^1(B, R^1q_*\Z(1))\to H^1(B, R^1q_*\mc{O})\to \dots$$ Recall that $H^1(\mc{A}^\vee, \Z)=H_1(\mc{A},\Z)$, and since everything here is a $K(\pi, 1)$-space, we know that $H^1(B, R^1q_*\Z(1))$ and $H^1(\pi_1(B), H_1(\mc{A}_b,\Z))$ are free abelian groups of the same rank. We claim that there exists an explicit isomorphism $F\colon H^1(\pi_1(B), H_1(\mc{A}_b,\Z))\to H^1(B, R^1q_*\Z(1))$ such that the following diagram commutes: 
\[\begin{tikzcd}
	{H^0(B, \mc{A})} & {H^1(\pi_1(B), H_1(\mc{A}_b,\Z))} \\
	{H^1(B, R^1q_*\Z(1))}
	\arrow["{\sec_{\text{ab}}}"', from=1-1, to=1-2]
	\arrow["\psi", from=1-1, to=2-1]
	\arrow["F", dashed, from=1-2, to=2-1]
\end{tikzcd}\]
To construct this isomorphism $F$, first observe that there is a natural inclusion map $\iota\colon H^0(B,\mc{A})\to H^0(B,(\mc{A})^{\text{cont}})$, where $(\mc{A})^{\text{cont}}$ is the sheaf of continuous sections from $B$ to $\mc{A}$, and the map $\sec_{\text{ab}}$ evidently factors through $\iota$. Now we claim that $\psi$ factors through $\iota$ as well. Consider the following commutative diagram of short exact sequences: 
\[\begin{tikzcd}
	0 & {R^1q_*\Z(1)} & {R^1q_*\mc{O}} & {\mc{A}} & 0 \\
	0 & {R^1q_*\Z(1)} & {(R^1q_*\mc{O})^{\text{cont}}} & {(\mc{A})^{\text{cont}}} & 0
	\arrow[from=1-1, to=1-2]
	\arrow[from=1-2, to=1-3]
	\arrow["{=}", from=1-2, to=2-2]
	\arrow[from=1-3, to=1-4]
	\arrow[from=1-3, to=2-3]
	\arrow[from=1-4, to=1-5]
	\arrow[from=1-4, to=2-4]
	\arrow[from=2-1, to=2-2]
	\arrow[from=2-2, to=2-3]
	\arrow[from=2-3, to=2-4]
	\arrow[from=2-4, to=2-5]
\end{tikzcd}\]
Taking long exact sequence in cohomology, we get the following commutative diagram
\[\begin{tikzcd}
	\dots & {H^0(B, \mc{A})} & {H^1(B, R^1q_*\Z(1))} & {H^1(B, R^1q_*\mc{O})} & \dots \\
	\dots & {H^0(B,(\mc{A})^{\text{cont}})} & {H^1(B, R^1q_*\Z(1))} & {H^1(B, (R^1q_*\mc{O})^{\text{cont}})}=0 & 
	\arrow[from=1-1, to=1-2]
	\arrow["\psi", from=1-2, to=1-3]
	\arrow["\iota", from=1-2, to=2-2]
	\arrow[from=1-3, to=1-4]
	\arrow["=", from=1-3, to=2-3]
	\arrow[from=1-4, to=1-5]
	\arrow[from=1-4, to=2-4]
	\arrow[from=2-1, to=2-2]
	\arrow["{\psi_{\text{cont}}}", from=2-2, to=2-3]
	\arrow[from=2-3, to=2-4]
\end{tikzcd}\]
This shows that $\psi$ factors through $\iota$. Therefore, it's enough to construct an isomorphism $F$ such that the following diagram commutes: 
\[\begin{tikzcd}
	{H^0(B, (\mc{A})^{\text{cont}})} & {H^1(\pi_1(B), H_1(\mc{A}_b,\Z))} \\
	{H^1(B, R^1q_*\Z(1))}
	\arrow["{(\sec_{\text{ab}})_{\text{cont}}}"', from=1-1, to=1-2]
	\arrow["\psi_{\text{cont}}", from=1-1, to=2-1]
	\arrow["F", dashed, from=1-2, to=2-1]
\end{tikzcd}\]

Furthermore, since $(R^1q_*\mc{O})^{\text{cont}}$ is a fine sheaf, we know $H^1(B,(R^1q_*\mc{O}^{\text{cont}}))=0$ and therefore the map $\psi_{\text{cont}}$ is surjective. On the other hand, because $p\colon \mc{A}\to B$ is a map of $K(\pi,1)$-spaces, we know that the map $(\sec_{\text{ab}})_{\text{cont}}$ is also surjective. This suggests the following definition of the map $F$: for any cohomology class $[s]\in H^1(\pi_1(B), H_1(\mc{A}_b,\Z))$, represent it by a continuous section $s\colon B\to \mc{A}$, and define $F([s])=\psi_{\text{cont}}(s)$.
\begin{lemma}\label{sheaf cohomology and group cohomology}
	The map $F$ constructed as above is well-defined and is an isomorphism which makes the desired diagram commute.
\end{lemma}
\begin{proof}
	Note that by construction, if $F$ is well-defined, it automatically makes the desired diagram commute, and is automatically a surjective group homomorphism. Since any surjection between free abelian groups of the same rank is automatically an isomorphism, it's enough to show that $F$ is well-defined.
	
	To show that $F$ is well-defined, we need to show that given any pair of continuous section $s_1$ and $s_2$, if they are homotopic, then $\psi_{\text{cont}}(s_1)=\psi_{\text{cont}}(s_2)$, and if they induce conjugate group theoretic splittings, then we also have $\psi_{\text{cont}}(s_1)=\psi_{\text{cont}}(s_2)$. To this end, we need to use an explicit description of the map $\psi_{\text{cont}}$. It's well-known that $H^1(B, R^1q_*\Z(1))$ parametrizes $R^1q_*\Z(1)$-torsors on $B$ \cite[Lemma 21.4.3]{Stacks}, and by unwinding the proof of this fact (e.g. the one given in \cite{Stacks}), we see that $\psi_{\text{cont}}$ admits the following description: let $s\colon B\to \mc{A}$ be a continuous section. Then $\psi_{\text{cont}}(s)$ is the isomorphism class of $R^1q_*\Z(1)$-torsor $\mc{F}\subset (R^1q_*\mc{O})^{\text{cont}}$ defined as the subsheaf of sections in $(R^1q_*\mc{O})^{\text{cont}}$ that maps to $s$ via the surjection $R^1q_*\mc{O}\to \mc{A}^\vee$.
	
	Now suppose $s_1$ and $s_2$ are homotopic continuous sections and let $\mc{F}_1$ and $\mc{F}_2$ be the $R^1q_*\Z(1)$-torsors they induce. Explicitly, over some trivializing open subset $U$, we know that $\mc{F}_i(U)$ is the set of sections $s\colon U\to U\times H^1(\mc{A}_b^\vee,\mc{O})$ which after composing with the map $U\times H^1(\mc{A}_b^\vee,\mc{O})\to U\times \mc{A}_b$ becomes $s_i|_U$, $i=1,2$. Note that the map $U\times H^1(\mc{A}_b^\vee,\mc{O})\to U\times \mc{A}_b$ is a covering map and hence we may lift the homotopy between $s_1$ and $s_2$ to a unique homotopy between sections of $\mc{F}_1(U)$ and $\mc{F}_2(U)$. In particular, $\mc{F}_1(U)\cong \mc{F}_2(U)$. Since the isomorphisms over each open is obtained by lifting the same homotopy between $s_1$ and $s_2$, they must glue to an isomorphism $\mc{F}_1\cong \mc{F}_2$. 
	
	Finally, suppose that $s_1$ and $s_2$ induce conjugate group theoretic splittings. This means that they are conjugate via an element of $H_1(\mc{A}_b,\Z)=H^1(\mc{A}^\vee,\Z)$, and so they are related via the Deck transformation coming from $R^1q_*\Z(1)$, and so by definition, the two torsors are isomorphic, as desired. 
\end{proof}
\begin{corollary}\label{Cor: sec_ab and psi}
	The map $\sec_{\text{ab}}$ is injective and/or surjective if and only if the map $\psi$ is injective and/or surjective.
\end{corollary}

Using this corollary, we prove the following proposition:
\begin{proposition}
	Let $p\colon \mc{A}\to B$ be an abelian scheme. Then the section map $\sec_{\text{ab}}\colon H^0(B, \mc{A})\to H^1(\pi_1(B), H_1(\mc{A}_b,\Z))$ is injective if and only if the monodromy action of $\pi_1(B)$ on $H^1(\mc{A}_b,\Z)$ has no invariant factors. 
\end{proposition}
\begin{proof}
	First observe that since $\mc{A}_b$ and $\mc{A}_b^\vee$ are isogenous, the monodromy action of $\pi_1(B)$ on $H^1(\mc{A}_b,\Z)$ having no invariants is the same as the monodromy action on $H^1(\mc{A}_b^\vee,\Z)$ having no invariants.
	
	Suppose the monodromy action has no invariants. Then by Corollary \ref{Cor: sec_ab and psi}, it's enough to show that $H^0(B, R^1q_*\mc{O})=0$. Consider the Higgs bundle associated to variation of Hodge structure $R^1q_*\Z$: $$\mc{E}:=q_*\omega_{\mc{A}^\vee/B}\oplus R^1q_*\mc{O}\xrightarrow{\theta}q_*\omega_{\mc{A}^\vee/B}\oplus R^1q_*\mc{O}\oplus \omega_{B},$$ where the Higgs field $\theta$ is defined by the following two maps $$\begin{aligned}
		&q_*\omega_{\mc{A}^\vee/B}\xrightarrow{\nabla}R^1q_*\mc{O}\otimes \omega_B\\
		&R^1q_*\mc{O}\xrightarrow{\text{zero map}} q_*\omega_{\mc{A}^\vee/B}\oplus R^1q_*\mc{O}\otimes \omega_B
	\end{aligned}$$ Here $\nabla$ is induced by the Gauss-Manin connection on the flat bundle associated to $R^1q_*\Z$.
	
	Now if $H^0(B, R^1q_*\mc{O})\neq 0$, then $\mc{O}_B$ maps into $R^1q_*\mc{O}$ and hence $(\mc{O}_B,0)$ will be a sub-Higgs bundle of $(\mc{E},\theta)$. On the other hand, by a theorem of Simpson \cite[Theorem 1]{Sim91}, we know that the $(\mc{E},\theta)$ is polystable, and so $(\mc{O},0)$ must be one of the irreducible factors of $(\mc{E},\theta)$ and hence the trivial representation should appear as a sub-representation of the monodromy representation, contradicting the assumption that the monodromy has no invariants.
	
	On the other hand, suppose that the map $\psi$ is injective. Then we see that we must have an isomorphism between $H^0(B, R^1q_*\Z(1))$ and $H^0(B, R^1q_*\mc{O})$. However, the former is a discrete group and the latter is a $\C$-vector space. Hence they are isomorphic if and only if both are $0$. In particular, $H^0(B, R^1q_*\Z(1))=0$ and the monodromy representation has no invariants.
\end{proof}
We can immediately deduce the following corollary for family of Jacobians associated to a Kodaira fibration:
\begin{corollary}
\label{Cor: injectivity of phi^ab}
Let $f\colon X\to B$ be a Kodaira fibration and $p\colon \Pic^0_{X/B}\to B$ be the corresponding family of Jacobians. Then $$\sec_{\text{ab}}\colon H^0(X,\Pic^0_{X/B})\to H^1(\pi_1(B),H_1(X_b,\Z))$$ is injective if and only if the associated monodromy action on $H^1(X_b,\Z)$ has no invariants.
\end{corollary}
Furthermore, the abelian version of the section question is also related to the original topological section question for Kodaira fibrations:
\begin{corollary}
\label{injectivity of top. section question}
Let $f\colon X\to B$ be a Kodaira fibration whose monodromy action on $H^1(X_b,\Z)$ has no invariants, then the corresponding map $$\{\text{algebraic sections to }f\colon X\to B\}\to \{\text{sections of (\ref{top. SES})}\}/\text{conjugation}$$ 
is injective.
\end{corollary}
\begin{proof}
If $f\colon X\to B$ has no sections, then the statement is trivially true so let's assume that we have a fixed section $s_0\colon B\to X$. Then we may define a $B$-morphism $h\colon C\to \Pic^0_{X/B}$ which maps $x\in X_b$ to the divisor class $[s_0(f(x))-x]$. Note that $h$ is injective, as it's just the Abel-Jacobi map on each fiber.

This section also allows us to identify $\Pic^0_{X/B}$ and $\Pic^1_{X/B}$ and so we have the following commutative diagram:
\[\begin{tikzcd}
	1 & {\pi_1(X_b)} & {\pi_1(X)} & {\pi_1(B)} & 1 \\
	0 & {H_1(X_b,\Z)=\pi_1(X_b)^{\text{ab}}} & {\pi_1(\Pic^0_{X/B})} & {\pi_1(B)} & 1
	\arrow[from=1-1, to=1-2]
	\arrow[from=1-2, to=1-3]
	\arrow[from=1-2, to=2-2]
	\arrow[from=1-3, to=1-4]
	\arrow[from=1-3, to=2-3]
	\arrow[from=1-4, to=1-5]
	\arrow["{=}", from=1-4, to=2-4]
	\arrow[from=2-1, to=2-2]
	\arrow[from=2-2, to=2-3]
	\arrow[from=2-3, to=2-4]
	\arrow[from=2-4, to=2-5]
\end{tikzcd}\]
Let $s$ and $s'$ be two distinct algebraic sections of $f\colon X\to B$. By post-composing with $h$ and using the injectivity of $h$, we get two distinct sections $\tilde{s}$ and $\tilde{s}'$ of $\pi\colon \Pic^0_{X/B}\to B$. If $s$ and $s'$ are  conjugate via some element $g\in \pi_1(X_b)$, then $\tilde{s}$ and $\tilde{s}'$ must be conjugate via the image of $g$ in $H_1(X_b,\Z)$, contradicting Corollary \ref{Cor: injectivity of phi^ab}. Hence, the map from algebraic sections to group theoretic splittings modulo conjugation is injective as desired.
\end{proof}

\subsection{Family of Jacobians and non-surjectivity}
In this subsection, we come back to the case of family of Jacobians $p\colon \Pic^0_{X/B}\to B$ associated to a Kodaira fibration $f\colon X\to B$. Note that since Jacobians of curves are principally polarized, they are self-dual and therefore we will henceforth identify $p\colon \Pic^0_{X/B}\to B$ and $q\colon (\Pic^0_{X/B})^\vee\to B$. The goal of this subsection is to show that if $f\colon X\to B$ admits an algebraic section, then the map $\sec_{\text{ab}}\colon H^0(B, \Pic^0_{X/B})\to H^1(\pi_1(B), H_1(\mc{A}_b,\Z))$ is never surjective. 

\begin{lemma}
Suppose that $f\colon X\to B$ is a Kodaira fibration with an algebraic section. Then the map $\sec_{\text{ab}}$ is not surjective if and only if the map $\phi\colon H^1(B, R^1f_*\Z(1))\to H^1(B, R^1f_*\mc{O}_X)$ induced by the inclusion $R^1f_*\Z(1)\to R^1f_*\mc{O}_X$ is non-zero.
\end{lemma}
\begin{proof}
	By Corollary \ref{Cor: sec_ab and psi}, we see that the map $\sec_{\text{ab}}$ is surjective if and only if the map $H^1(B, R^1p_*\Z(1))\to H^1(B, R^1p_*\mc{O}_{\Pic^0_{X/B}})$ is non-zero. Now just like before, any algebraic section of $f\colon X\to B$ induces an inclusion $h\colon X\to \Pic^0_{X/B}$ and hence morphisms of sheaves $R^1p_*\Z(1)\to R^1f_*\Z(1)$ and $R^1p_*\mc{O}_{\Pic^0_{X/B}}\to R^1f_*\mc{O}_X$. By checking locally, one can verify that these two morphisms of sheaves are in fact isomorphisms and therefore we can identify the map $H^1(B, R^1p_*\Z(1))\to H^1(B, R^1p_*\mc{O}_{\Pic^0_{X/B}})$ with the map $H^1(B, R^1f_*\Z(1))\to H^1(B, R^1f_*\mc{O}_X)$ as desired.
\end{proof}

Then it's enough to understand the map $\phi\colon H^1(B, R^1f_*\Z(1))\to H^1(B, R^1f_*\mc{O}_X)$. The first step is to compute the degree of the vector bundle $R^1f_*\mc{O}_X$. This computation is probably known to experts but we could not find a reference:
\begin{lemma}\label{GRR computation}
	The vector bundle $R^1f_*\mc{O}_X$ is of negative degree.
\end{lemma}
\begin{proof}
Let $f\colon X\to B$ be a Kodaira fibration such that the fiber $X_b$ has genus $g$ and the base $B$ has genus $h$. By Grothendieck-Riemann-Roch, we know that $$\ch(f_!(\mathcal{O}_X))=f_*(\ch(\mathcal{O}_X)\cdot \td_{X/B}),$$ where $\ch$ denotes the Chern character, and $\td_{X/B}$ is the Todd class of the relative tangent bundle $\mathcal{T}_{X/B}$. Since we are in the relative curve setting, we know that the higher derived pushforward $R^if_*(\mathcal{O}_X)$ vanishes for all $i\geq 2$, so we can rewrite the the left hand side of the equation to get
$$\begin{aligned}
    \ch(f_!(\mathcal{O}_X))&=\ch(f_*\mathcal{O}_X)-\ch(R^1f_*(\mathcal{O}_X))\\
    &=\ch(\mathcal{O}_B)-\ch(R^1f_*(\mathcal{O}_X))\\&=1-(\rk(R^1f_*(\mathcal{O}_X))+c_1(R^1f_*(\mathcal{O}_X))+\dots).
\end{aligned}$$ 

On the other hand, since $\ch(\mathcal{O}_X)=1$, we see that the right hand side is simply $$\begin{aligned}
f_*\td_{X/B}&=f_*(\td(\mathcal{T}_{X/B}))\\
&=f_*\left(1+\frac{c_1(\mathcal{T}_{X/B})}{2}+\frac{c_1^2(\mathcal{T}_{X/B})+c_2(\mathcal{T}_{X/B})}{12}+\dots\right).
\end{aligned}$$ As $\mathcal{T}_{X/B}$ is a line bundle, we know that it has no higher Chern classes. It follows then $$\deg R^1f_*(\mathcal{O}_X)=c_1(R^1f_*(\mathcal{O}_X))=-f_*\left(\frac{c_1^2(\mathcal{T}_{X/B})}{12}\right).$$ 

Thus, it's enough to compute $f_*(c_1^2(\mathcal{T}_{X/B}))$. Since $\mathcal{T}_{X/B}=\Omega_{X/B}^\vee$, we know that $c_1^2(\mathcal{T}_{X/B})=c_1^2(\Omega_{X/B})$ so we can work with the relative differential. Consider the following short exact sequence $$0\to f^*\Omega_B\to \Omega_X\to \Omega_{X/B}\to 0.$$ By taking the wedge power, we get the following isomorphism $$\wedge^2\Omega_X\cong f_*\Omega_B\otimes \Omega_{X/B}.$$ Since $c_1$ is a group homomorphism, we know that $c_1(\wedge^2\Omega_X)=c_1(f_*\Omega_B)+c_1(\Omega_{X/B})$. Then $$\begin{aligned}c_1(\Omega_{X/B})^2&=(c_1(\wedge^2\Omega_X))^2-2c_1(\wedge^2\Omega_X)\cdot c_1(f^*\Omega_B)+(c_1(f^*\Omega_B))^2\\
&=K_X^2-2K_X\cdot c_1(f^*(K_B))+(c_1(f^*(K_B))^2)
\end{aligned}$$ where $K_X$ is a canonical divisor on $X$, and $K_B$ is a canonical divisor on $B$. Now since $c_1$ is functorial and $f^*$ is a ring homomorphism, we know that $$(c_1(f^*(K_B)))^2=f^*(c_1(K_B)^2).$$ Because $B$ is a curve, this has to vanish. It follows then $$(c_1(\Omega_{X/B}))^2=K_X^2-2K_X\cdot f^*K_B$$ and hence $$f_*((c_1(\Omega_{X/B}))^2)=f_*(K_X^2)-2f_*(K_X\cdot f^*K_B).$$ Since we may compute degrees both before and after pushing-forward, we know that $f_*(K_B^2)=K_B^2$. To understand the last term, we first use projection formula to write it as $$2f_*(K_X\cdot f^*K_B)=2f_*(K_X)\cdot K_B.$$ Now $\deg f_*(K_X)=K_X\cdot X_b$, and because $K_X\cdot X_b=\deg K_X|_{X_b}=\deg K_{X/B}|_{X_b}$ for any generic fiber $X_b$, we see that $K_X\cdot X_b=2g-2$,. Now since $K_B$ has degree $2h-2$, we see that $$f_*((c_1(\Omega_{X/B}))^2)=K_X^2-8(g-1)(h-1).$$ Hence, we have $$\deg R^1f_*\mathcal{O}_X=c_1(R^1f_*\mathcal{O}_X)=\frac{-K^2_X+8(g-1)(h-1)}{12}=\frac{-K^2_X+2\chi_X}{12},$$ where $\chi_X$ is the Euler characteristics of $X$. Finally, by the signature formula of Hirzebruch, Atiyah and Singer, we know that the signature $\sigma(X)$ of $X$ is precisely given by $$\sigma(X)=\frac{1}{3}(K^2_X-2\chi_X).$$ Since Kodaira fibrations necessarily have positive signatures \cite[Corollary 42]{Cat17}, we may conclude that $\deg R^1f_*\mc{O}_X<0$.
\end{proof}
\begin{corollary}
\label{dim of H^1}
$\dim H^1(B, R^1f_*\mc{O}_X)>3$.
\end{corollary}
\begin{proof}
By Riemann-Roch, we know that $$\dim H^1(B, R^1f_*\mc{O}_X)=-\left(\deg R^1f_*\mc{O}_X+\rk(R^1f_*\mc{O}_X)(1-h)-\dim H^0(B,R^1f_*\mc{O}_X)\right).$$ Lemma \ref{GRR computation} says that $-\deg R^1f_*\mc{O}_X>0$. Furthermore, it's well known (e.g.\cite[Theorem 1.1]{Kas68}) that the base curve of a Kodaira fibration has genus at least $2$ and the fiber has genus at least $3$, so we know that $$\dim H^1(B, R^1f_*\mc{O}_X)>\rk(R^1f_*\mc{O}_X)(h-1)\geq 3\cdot 1=3,$$ as desired.
\end{proof}
\begin{proposition}
	The map $\phi\colon H^1(B, R^1f_*\Z(1))\to H^1(B, R^1f_*\mc{O}_X)$ is non-zero. 
\end{proposition}
\begin{proof}
	Let $\Phi\colon H^2(X, \Z(1))\to H^2(X, \mc{O}_X)$ be the map induced by the exponential map $\Z(1)\to \mc{O}_X$. We know that both $H^2(X, \Z(1))$ and $H^2(X, \mc{O}_X)$ are equipped with compatible Leray filtrations and the map $\phi \colon  H^1(X, R^1f_*\Z(1))\to H^1(X, R^1f_*\mc{O}_X)$ is induced by $\Phi$ between the associated graded pieces of the Leray filtrations. Because the map $f\colon X\to B$ admits a section, the Leray filtration splits, and so we may view $H^1(B,R^1f_*\Z(1))$ (resp. $H^1(B,R^1f_*\mc{O}_X)$) as a subgroup of $H^2(B,\Z(1))$ (as a subspace of $H^2(S,\mc{O}_X)$). As the map $\Phi$ factors through $H^2(X,\C)$, we have the following commutative diagram:
	\[\begin{tikzcd}
	{H^2(X,\Z(1))} & {H^2(X,\C)} & {H^2(X,\mc{O}_X)} \\
	{H^1(B,R^1f_*\Z(1))} && {H^1(B,R^1f_*\mc{O}_X)}
	\arrow[from=1-1, to=1-2]
	\arrow[from=1-2, to=1-3]
	\arrow[from=2-3, to=1-3]
	\arrow[from=2-1, to=1-1]
	\arrow["\phi"', from=2-1, to=2-3]
\end{tikzcd}\]

Note that the map $H^2(X,\C)\to H^2(X,\mc{O}_X)$ is the projection map induced by the Hodge decomposition on $H^2(X,\C)$ and in particular this map is surjective. Now as $H^0(B,R^2f_*\mc{O}_X)=H^2(B, f_*\mc{O}_X)=0$, we know that $H^2(X,\mc{O}_X)\cong H^1(B,R^1f_*\mc{O}_X)$. It follows from Corollary \ref{dim of H^1} that $\dim H^2(X,\mc{O}_X)>3$. On the other hand, we also have $\dim H^0(B, R^2f_*\Z(1))=\dim H^2(B, f_*\Z(1))=1$, so we see that $H^1(B, R^1f_*\Z(1))$ is a lattice inside $H^2(X, \C)$ which generates a subspace of codimension $2$.Then for the projection map from $H^2(X,\C)\to H^2(X,\mc{O}_X)$ to be surjective, its restriction to $H^1(B, R^1f_*\Z(1))$ must be non-zero and hence the map $\phi$ must be non-zero.
\end{proof}
This proves the main claim of this subsection:
\begin{corollary}
\label{surjectivity}
Let $f\colon X\to B$ be a Kodaira fibration with an algebraic section, and $p\colon \Pic^0_{X/B}\to B$. The map $\sec_{\text{ab}}$ is never surjective for these families of Jacobians.
\end{corollary} 

\begin{remark}
	In the case where $f\colon X\to B$ has an algebraic section, one may identify $\Pic^0_{X/B}$ with $\Pic^1_{X/B}$, and hence our results shows that the topological section question has a negative answer for $\Pic^1_{X/B}\to B$ in the case where the associated Kodaira fibration has a section. This in fact differs from the universal case: the universal family of moduli space of degree $1$ line bundles $p'\colon \Pic^1_{\mc{C}_g/\mc{M}_g}\to \mc{M}_g$ does satisfies the topological section question when $g\geq 3$ in a trivial way, i.e., there is no topological section to $p'$ (this is first proven by Morita when $g\geq 9$; see \cite[Corollary 3, Theorem 4]{Mor86}. The strengthened result is proven in \cite[Theorem 5.1.11]{LLSS}).
\end{remark}

\section{Weak section conjecture and counterexample to surjectivity}
In this section, we study the surjectivity of the topological section map $\sec_{\text{top}}$ for Kodaira fibrations $f\colon X\to B$. We first adapt an argument of Stix \cite[Theorem 31]{Sti10} to show that just like in the arithmetic case, the surjectivity problem can be reduced to the case of Kodaira fibrations with no algebraic sections. We've chosen to highlight the main ingredients that make this strategy work so we will first work in a fairly general setting.

Let $f\colon X\to B$ be a smooth projective map with connected fibers between smooth projective varieties. Assume that we have a short exact sequence of topological fundamental groups \begin{equation}\label{top. SES 2} 1\to \pi_1(X_b)\to \pi_1(X)\to \pi_1(B)\to 1\end{equation} where $X_b$ is the fiber over $b\in B$. In this general setting, we have the following two claims:
\begin{statement}[Weak topological section conjecture]
\label{Claim 1}
	The map $X\to B$ admits an algebraic section if and only if the associated short exact sequence of topological fundamental groups split.
\end{statement}
\begin{statement}[Surjectivity of top. section question]
\label{Claim 2}
	The section map $$\sec_{\text{top}}:\{\text{algebraic sections to }f:X\to B\}\to \{\text{splittings of (\ref{top. SES 2})}\}/\text{conjugation}$$ is surjective.
\end{statement}
We prove the following proposition:
\begin{proposition}
\label{Prop: weak=surjectivity}
	Assume that
	\begin{enumerate}
		\item $\pi_1(X_b)$ is residually finite.
		\item The set $X'(B)$ of algebraic sections of $X'\to B$ is finite for every finite \'{e}tale connected cover $X'\to X$ such that the composed map $X'\to B$ has connected fibers.
	\end{enumerate}
	
	Then Statement \ref{Claim 2} being true for $X\to B$ is equivalent to Statement \ref{Claim 1} being true for all finite \'{e}tale connected covers $X'$ of $X$ such that the composed map $X'\to B$ has connected fibers.
\end{proposition} 
To distinguish an algebraic section of $f\colon X\to B$ and a group theoretic section of (\ref{top. SES 2}), we will denote the former $s$ and the latter $x$. 
\begin{definition}[Neighbourhood of a section $x$]
	Let $x\colon \pi_1(B)\to \pi_1(X)$ be a group theoretic splitting of the short exact sequence of topological fundamental groups (\ref{top. SES}). Then a \textit{neighbourhood} of $x$ is a finite \'{e}tale connected cover $S'$ of $S$ such that $S'\to C$ has connected fibers and the finite index subgroup $\pi_1(S')\subset \pi_1(S)$ contains the image $x(\pi_1(C))$ of the section $x$.
\end{definition}
Note that given a neighbourhood of $x$, we get a lift of $x$ to a group theoretic section $x'\colon \pi_1(B)\to \pi_1(X')$ of the short exact sequence of fundamental groups associated to $X'\to B$. Furthermore, if we post-compose $x'$ with the natural inclusion map $\pi_1(X')\to \pi_1(X)$, we recover the section $x$ so one may alternatively define a neighbourhood of $x$ as a pair $(X',x')$ of finite \'{e}tale covers $S'$ with connected fibers over $B$ and a group theoretic section $x'$ which descends to $x$.

\begin{lemma}
\label{lemma: lifting algebraic sections}
	Let $x=x_s$ be a geometric section, i.e., it's induced by some algebraic section $s\colon B\to X$, then a neighbourhood of $x$ is the same as pair $(X',x_{s'})$, where $X'$ is a finite \'{e}tale connected cover of $X$ with connected fibers over $B$, and $x_{s'}$ is a group theoretic section induced by some algebraic section $s'\colon B\to X$ that is a lift of $s$.
\end{lemma}
\begin{proof}
	Recall that a finite \'{e}tale connected covers is the same as a finite set with a transitive $\pi_1(X)$ action. In this case, the finite set is given by the  set $\pi_1(X)/\pi_1(X')$ of cosets of $\pi_1(X')$. Using the section $x_s\colon \pi_1(B)\to \pi_1(X)$, we get an induced action of $\pi_1(B)$ on this set. Since $\pi_1(X')$ contains $\pi_1(B)$, it follows that this action has a fixed point. Therefore, the cover that corresponds to this action of $\pi_1(B)$ is disconnected and has a copy of $B$. Hence, we may lift the section $s$ to an algebraric section $s'\colon B\to X$. 
\end{proof}

Given a group theoretic splitting $x\colon \pi_1(B)\to \pi_1(X)$, let $X_x$ be the pro-\'{e}tale cover of $X$ defined by the limit of the projective system $(X'\to X)$, where $X'$ runs over all neighbourhoods of $x$. 
\begin{lemma}
\label{lemma: sections with the same neighbourhood}
	Let $x_1$ and $x_2$ be two group theoretic sections. Suppose $\pi_1(X_b)$ is residually finite. Then $X_{x_1}=X_{x_2}$ if and only if they are conjugate to each other.
\end{lemma}
\begin{proof}
	It's enough to show that $\pi_1(X_x)=x(\pi_1(C))$. This is the case since $X_{x_1}=X_{x_2}$ is equivalent to $\pi_1(X_{x_1})$ being conjugate to $\pi_1(X_{x_2})$. Hence, if $\pi_1(X_x)=x(\pi_1(B))$, then $\pi_1(X_{x_1})$ being conjugate to $\pi_1(X_{x_2})$ is equivalent to $x_1$ being conjugate to $x_2$.
	
	Now to see that $\pi_1(X_x)=s(\pi_1(B))$, consider all finite index subgroups $H$ of $\pi_1(X)$ containing $x(\pi_1(B))$. Observe that $\pi_1(X_x)=\bigcap H$ so it's enough to show that $x(\pi_1(B))=\bigcap H$. Since we have a section, $\pi_1(X)$ can be written as a semi-direct product $\pi_1(X)\cong \pi_1(X_b)\rtimes \pi_1(B)$. Let $N_i:=\bigcap_{[\pi_1(X_b):H]=i} H\subset \pi_1(X_b)$. Note that because $\pi_1(X_b)$ is finitely generated, it admits finitely many maps into the symmetric group $S_i$ and hence there are only finitely many index $i$ subgroups of $\pi_1(X_b)$. In particular, this intersection is finite and $N_i$ is again of finite index. Since every automorphism of $\pi_1(X_b)$ preserves the index of a subgroup, we see that $N_i$ is also characteristics. It follows that $N_ix(\pi_1(B))$ are finite index subgroups of $\pi_1(X)$. Since $\pi_1(X_b)$ is residually finite, we know that $\bigcap N_i$ is trivial. It follows that $N_ix(\pi_1(B))=x(\pi_1(B))$ and hence the intersection of all finite index subgroups of $\pi_1(X)$ containing $x(\pi_1(B))$ is $x(\pi_1(B))$ as desired.
\end{proof}
\begin{remark}
	The proof of this lemma is essentially a minor modification of Malcev's proof (for example, see \cite{MSE23}) that semi-direct products of residually finite, finitely generated groups are residually finite.
\end{remark}
Combining these two lemmas, we may give a characterization of group theoretic sections that come from algebraic geometry:
\begin{corollary}
	\label{corollary: characterization of geometric sections}
	A group theoretic section $x$ is conjugate to $x_s$ for some algebraic section $s\colon B\to X$ if and only if $s$ belongs to the image of the natural map $X_x(B)\to X(B)$, where $X_x(B)$ is the set of algebraic sections of $X_x\to B$ and $X(B)$ is the set of algebraic sections of $X\to B$.
\end{corollary}
\begin{proof}
	If $x\colon B\to X$ is an algebraic section, by lemma \ref{lemma: lifting algebraic sections}, it lifts to a compatible system of algebraic sections and hence it lifts to an algebraic section of $X_x\to B$.
	
	Conversely, if $s$ is in the image of $X_x(B)\to X(B)$, the section $x$ and $x_s$ have the same collection of neighbourhood and so are conjugate to each other by lemma \ref{lemma: sections with the same neighbourhood}.
\end{proof}

Now we are ready to prove Proposition \ref{Prop: weak=surjectivity}. 
\begin{proof}[Proof of Prop. \ref{Prop: weak=surjectivity}]
		First let's deduce the weak topological section conjecture for connected finite \'{e}tale covers with connected fibers over $B$ from the surjectivity of the section map for $f\colon X\to B$. Suppose there exists a connected finite \'{e}tale cover $X'\to X$ with connected fibers over $B$ and a group theoretic section $x'$. Then $x'$ descends to a group theoretic section from $\pi_1(B)\to \pi_1(X)$ and hence by the surjectivity, there exists an algebraic section from $B\to X$. By Lemma \ref{lemma: lifting algebraic sections}, we may lift this to an algebraic section from $B$ to $X'$, and hence the weak topological section question holds for $X'\to B$.

	For the other direction, by Corollary \ref{corollary: characterization of geometric sections}, it's enough to show that $X_x(B)=\varprojlim X'(B)$ is non-empty, where $X'(B)$ is the set of algebraic sections of $X'\to B$. Since every neighbourhood $X'$ has a topological section by definition, it follows from the weak topological section conjecture that $X'(B)$ is non-empty. It's finite by assumption, and therefore it's a non-empty set with a compact Hausdorff topology. Then such projective limit is always non-empty as desired.
\end{proof}

Since Kodaira fibrations satisfy the assumptions of Prop. \ref{Prop: weak=surjectivity}, we may deduce the following corollary:
\begin{corollary}
\label{Cor: weak=surjectivity for Kodaira fibrations}
	Let $f\colon X\to B$ be a Kodaira fibration. 	The surjectivity of the topological section map $\sec_{\text{top}}$ is equivalent to the weak topological section conjecture for all connected finite \'{e}tale cover $X'$ of $X$ such that $X'\to B$ has connected fibers.
\end{corollary}
\begin{proof}
	The fact that fundamental groups of surfaces are residually finite is the main theorem of \cite{Hem72} and the fact that the set of algebraic section is finite follows from the geometric Mordell conjecture proven by Manin \cite{Man63}.
\end{proof}
\begin{corollary}
	There are counterexamples to the surjectivity part of the topological section question with no algebraic sections.
\end{corollary}
\begin{proof}
	This is now a direct consequence of Corollary \ref{Cor: weak=surjectivity for Kodaira fibrations} and the counterexamples constructed by Lee and Serv\'{a}n in Appendix \ref{section: appendix A}.
\end{proof}

In fact, Proposition \ref{Prop: weak=surjectivity}, will hold in general for any non-isotrivial family $X$ of smooth projective curves over some smooth projective base $B$. More specifically, we have the following result
\begin{corollary}
	Let $f\colon X\to B$ be a smooth projective family of curves corresponding to some non-constant map from $B$ to $\mc{M}_g$. Then the surjectivity of the section map is equivalent to the weak topological section conjecture.
\end{corollary}
\begin{proof}
	It's enough to verify that such a family of curves $X$ has only finitely many algebraic sections. If $X\to B$ has infinitely many algebraic sections, note that the locus where these infinitely many algebraic sections agree is a countable union of closed subvarieties of $B$, and hence we may find a smooth proper curve $C\subset B$ such that the restriction of all these sections are distinct. Furthermore, we may also choose $C$ such that the map from $B$ to $\mc{M}_g$ restricts to a non-constant map on $C$. This then gives us a Kodaira fibration with infinitely many algebraic sections, contradicting the geometric Mordell Conjecture \cite{Man63}.
\end{proof}

\section{Hodge theoretic section question and family of curves}\label{section: Hodge theoretic version}
In this section, we turn to the Hodge theoretic section question by replacing the topological fundamental group to a Hodge theoretic invarient, i.e. the category $\VMHS_\Z$, which is sensitive to the underlying algebraic structure of a given variety. First, we record some preliminary facts about this category and give a precise formulation of the Hodge theoretic section question. Next, we prove an injectivity for smooth projective families of curves of genus at least $1$.

\subsection{Preliminary and formulation of the question}
Let $X$ be a smooth connected variety over $\C$. Let $\VMHS_\Z(X)$ be the category of admissible, graded-polarizable, $\Z$-variation of mixed Hodge structures (VMHS) over $X$. For a precise definition of such an object, see \cite{Kas86}. Note that this is an abelian tensor category. The unit object in this category is the constant variation of mixed Hodge structure $\underline{\Z(0)}$ and by trivial objects, we mean direct sums of the unit objects.

Suppose $X$ and $Y$ are two smooth complex varieties. Any functor $F\colon \VMHS_\Z(X)\to \VMHS_\Z(Y)$ is assumed to be exact, additive $\otimes$-functor. 
\begin{remark}
	Admissibility is a technical condition on the behavior of a variation of mixed Hodge structure at infinity that will not play a role in the proof of our main injectivity result (Prop. \ref{Proposition: Hodge theoretic injectivity}). The main point is that it ensures the variation of mixed Hodge structure to have some nice properties, which all variations of mixed Hodge structure of geometric origin have. Any VMHS that comes from geometry is admissible. If one wants to ignore this technical point, one may assume that all the spaces in this section are compact, in which case all graded-polarizable variation of $\Z$-mixed Hodge structures are automatically admissible.
\end{remark}

Let $f\colon X\to B$ be a smooth projective morphism with connected fibers between two smooth connected varieties, and let $f^*\colon\VMHS_\Z(B)\to \VMHS_\Z(X)$ be the pullback functor induced by $f$. Note that it satisfies the conditions we listed above. Furthermore, if $s\colon B\to X$ is an algebraic section to $f$, we get a functor $s^*\colon \VMHS_\Z(X)\to \VMHS_\Z(B)$ such that $s^*\circ f^*=\id_{\VMHS_\Z(B)}$. We can make a formal definition:
\begin{definition}
\label{section functor}
	A functor $F\colon \VMHS_\Z(X)\to \VMHS_\Z(Y)$ is a section to $f^*$ if $F\circ f^*$ is isomorphic to the identity functor on $\VMHS_\Z(Y)$.
\end{definition}
Now we formulate the Hodge theoretic section question:
\begin{question}[Hodge theoretic section question] \hfill
\label{Question:Hodge theoretic section question}
\begin{enumerate}
	\item (injectivity) If $s_1,s_2$ are two distinct algebraic sections to $f\colon X\to B$, can the functors $s_1^*$ and $s_2^*$ be isomorphic?
	\item (surjectivity) Suppose that $F\colon \VMHS_\Z(X)\to \VMHS_\Z(B)$ is a functor which is a section to $f^*$. Then can we find an algebraic section $s\colon B\to X$ such that $F$ is isomorphic to $s^*$?
\end{enumerate}
\end{question}

\subsection{Injectivity for family of curves}
In this subsection, we show that the injectivity part of the Hodge theoretic section question has a positive answer in the case of smooth projective family of curves. Let $f\colon X\to B$ be a family of smooth projective curves of genus $g\geq 2$ over some smooth connected base $B$ (which is not assumed to be proper), and $\phi\colon B\to \mc{M}_g$ the corresponding map into $\mc{M}_g$. We have the following proposition:
\begin{proposition}
	\label{Proposition: Hodge theoretic injectivity}
	Let $f\colon X\to B$ be a family of curves as above. Then for any pair of algebraic sections $s_1,s_2\colon B\to X$, if $s_1^*$ is isomorphic to $s_2^*$ as functors from $\VMHS_\Z(X)\to \VMHS_\Z(B)$, then $s_1=s_2$. 
\end{proposition}

To prove this proposition, we need to find a graded-polarizable, admissible $\Z$-variation of mixed Hodge structure on $X$ whose associated period map is injective (or at least injective on each fiber). We do so by using the canonical variation of mixed Hodge structure of Hain and Zucker. We first recall some definitions and facts.

Let $X$ be a smooth algebraic variety over $\C$ endowed with the analytic topology, and let $PX$ be the space of piecewise-smooth paths in $X$ endowed with the compact open topology. The free path fibration $p\colon PX\to X\times X$ is defined as $$\begin{aligned}
	p\colon PX&\to X\times X\\
	\gamma&\mapsto (\gamma(0),\gamma(1))
\end{aligned}$$
Denote the $P_{x,y}$ the fiber of $p\colon PX\to X\times X$ over the point $(x,y)$. Now there is an isomorphism $H_0(P_{x,x},\Z)\cong \Z[\pi_1(X,x)]$. Let $J_x$ be the augmentation ideal of the group ring $\Z[\pi_1(X,x)]$. Note that $H_0(P_{x,y})$ carries a canonical left $\Z[\pi_1(X,x)]$-module structure, so we get an induced filtration $J^\bullet$ by the augmentation ideal $J_x$. 

\begin{proposition-definition}[r-th canonical VMHS, Prop. 4.20 + Def. 4.21 of \cite{HZ87}]
	Let $X$ be a smooth algebraic variety over $\C$ and $x\in X$ a fixed point. Then there exists a graded-polarizable variation $\mc{J}_x$ of mixed Hodge structure on $X$ such that for any $y \in X$, $$\mc{J}_{x,y}:=(\mc{J}_x)_y=H_0(P_{x,y},\Z)/J^{r+1}.$$ 
\end{proposition-definition}

We will in particular be interested in the case where $r=1$. In this case, we have an extension of mixed Hodge structures \cite[Prop. 5.39]{HZ87} $$0\to H_1(X,\Z)\to H_1(X,\{x,y\})\to \Z(0)\to 0$$ In particular, when $x\neq y$, we just have $H_0(P_{x,y},\Z)/J^2\cong H_1(X,\{x,y\})$. We have the following proposition which classifies such extensions:
\begin{proposition}\cite{Car80}
	Extensions of this form is classified by the Albanese $\Alb(X)$ of $X$, and the map $y\mapsto H_1(X,\{x,y\})$ agrees with the Albanese mapping with basepoint $x$: $$\begin{aligned}
		\alpha_x\colon X&\to \Alb(X):=F^1H^1(X)^\vee/H_1(X,\Z)\\
		y&\mapsto \left(\omega\mapsto\int_\gamma\omega\right)
	\end{aligned}$$ where $\gamma$ is any path from $x$ to $y$. 
\end{proposition} 
\begin{theorem}\cite[Cor. 5.40]{HZ87}
	This period map $\alpha_x$ agrees with the period map for the $1$-st canonical VMHS.
\end{theorem}

\begin{corollary}
\label{cor: Injectivity of Abel-Jacobi}
	When $(X,x)$ is a curve, the $1$-st canonical VMHS on $X$ with base point $x$ has injective period map.
\end{corollary}

Now we may proceed to the proof of Proposition \ref{Proposition: Hodge theoretic injectivity}.
\begin{proof}[Proof of Prop. \ref{Proposition: Hodge theoretic injectivity}]
	If $f\colon X\to B$ has no section, then the claim is trivially true, so we may without loss of generality and assume that we have an algebraic section $s_0\colon B\to X$. Then as before we get a commutative diagram  
$$
\begin{tikzcd}
X\arrow[r,"h"]\arrow[dr,"f"]&\Pic^0_{X/B}\arrow[d,"\pi"]\\
& B
\end{tikzcd}
$$
The fibers of $\pi\colon \Pic^0_{X/B}\to B$ is $\Jac(X_b)=\Alb(X_b)$, which we may view as a mixed period domain. In particular, $\Pic^0_{X/B}$ carries a universal variation of mixed Hodge structure $\mc{U}$ such that for a given point $p\in \Jac(X_b)$, $\mc{U}|_p$ is the extension class in $\Ext^1(\Z(0),H_1(X_b))$ corresponding to $p$. Now pulling back $\mc{U}$ along $h$, we get a variation of mixed Hodge structure $\mc{J}:=h^*\mc{U}$ on $\mc{X}$, whose period map factors through $h$, and which, when restricting to the fiber $X_b$, agrees with $J_{s_0(b),y}$. By Corollary \ref{cor: Injectivity of Abel-Jacobi}, we know that $\mc{V}$ is injective on each fiber.

Therefore, if $s_1^*$ is isomorphic to $s_2^*$ as functors, then for all $b\in \mc{B}$, $s_1^*\mc{V}|_b\cong s_2^*\mc{V}|_b$ or equivalently, $\mc{V}|_{s_1(b)}\cong \mc{V}|_{s_2(b)}$. By the injectivity of the period map, we see that $s_1(b)=s_2(b)$. Therefore, the desired proposition follows immediately if one can verify that $\mc{J}$ is admissible.
\end{proof}
\begin{lemma}
	The $\Z$-VMHS $\mc{J}$ constructed in the proof above is admissible.
\end{lemma}
This lemma, again, is automatically true when $X$ is proper. It was proven in \cite{Hai87-2} and we provide a sketch of a proof that's motivated by Beilinson-Deligne-Goncharov's construction of mixed Hodge structures on truncated fundamental groups \cite{DG05}. For a detailed description of this construction, see \cite[section 3.6]{BF}.
\begin{proof}[Sketch of a proof]
	We claim that this variation of mixed Hodge structures comes from geometry as it's the cohomology of a family of cosimplicial schemes. Since any variation of mixed Hodge structure that comes from geometry is admissible, this proves the desired claim. The main idea is to run the Beilinson-Deligne-Goncharov's construction of the mixed Hodge structure on $\Z[\pi_1(X,x,y)]/J^{r+1}$ in families.
	
	Consider the fiber product 
\[\begin{tikzcd}
	{X\times_{B}X} & {X} \\
	{X} & {B}
	\arrow["{p_2}",from=1-1, to=1-2]
	\arrow["{p_1}",from=1-1, to=2-1]
	\arrow[from=1-2, to=2-2]
	\arrow[from=2-1, to=2-2]
\end{tikzcd}\] We get a fibration $\varphi\colon X\times_B X \to B$, where over each point $b\in\mc{B}$, the fiber is given by the product of curves $\varphi^{-1}(b)=X_b\times X_b$. Let $\mc{Z}_0$ be the image of the diagonal map $\Delta\colon X\to X\times_B X$. Let $D$ be the image of the fixed section $s_0\colon B\to X$, and let $\mc{Z}_1$ be the preimage of $D$ in $X\times_B X$ under the second projection map $p_2$. Note that $\mc{Z}_0\cap \pi^{-1}(b)$ is the closed subset $Z_0\subset X_b\times X_b$ defined by $\{x_1=x_2\}$, where the $x_i$ are coordinates of $X_b\times X_b$ and $\mc{Z}_1\cap \pi^{-1}(b)$ is the closed subset $Z_1\subset X_b\times X_b$ defined by $\{x_1=s_0(b)\}$ 

Let $\underline{\Z}_{\mc{Z}_i}$ be the extension by zero of the constant sheaf on $\mc{Z}_i$ along the natural inclusion map. We can define the following complex $$\mc{K}_s\colon 0\to \underline\Z\to \underline\Z_{\mc{Z}_0}\oplus \underline\Z_{\mc{Z}_1}\to 0,$$ where the map $\underline{\Z}\to \underline\Z_{\mc{Z}_0}\oplus \underline\Z_{\mc{Z}_1}$ is given by the alternating sum of the natural restriction map. Note that if we restrict this complex to $\pi^{-1}(b)$, we recover the complex of sheaves on $X_b\times X_b$ used in Beilinson-Deligne-Goncharov's construction $${}_\bullet\mc{K}_{s(b)}\langle 1\rangle\colon 0\to \underline{\Z}\to \underline{\Z}_{Z_0}\oplus \underline{\Z}_{Z_1}\to 0.$$ Now the desired variation of mixed Hodge structure agrees with the variation of mixed Hodge structure defined on the local system $R^1(p_1)_*(\mc{K}_s)$ on $X$, whose fiber at $y\in S_b$ is given by the hypercohomology $\mathbb{H}^1(X_b,{}_y\mc{K}_{s(b)}\langle 1\rangle)$, which agrees with $H^1(X_b,\{s(b),y\})$ when $s(b)\neq y$. When $s(b)=y$, this hypercohomology becomes the split extension of $H^1(X_b,{s(b)})$ by $\underline{\Z(0)}$. 
\end{proof}
\begin{corollary}
\label{Cor. Hodge theoretic inject. for Kod. Fib}
	Let $f\colon X\to B$ be a Kodaira fibration, and $s_1,s_2$ two distinct algebraic sections of $f$. Then $s_1^*$ is not isomorphic to $s_2^*$.
\end{corollary}

\newpage
\appendix

\section{Proof of Proposition \ref{Prop: surjectivity of top. section conjecture} part (2)}\label{section: appendix A}
\smallskip
\begin{center}by \textsc{Seraphina Eun Bi Lee and Carlos A. Serv\'{a}n}\end{center}

The purpose of this appendix is to prove Proposition \ref{Prop: surjectivity of top. section conjecture} part (2). The construction given here is an adaptation of our work \cite[Remark 6.4]{LS24}. Throughout, let $f: X \to B$ be a Kodaira fibration of fiber genus $g$ with an algebraic section $s: B \to X$.

Consider a double cover $B' \to B$ branched over two points $p_1, p_2 \in B$ and the base change 
\[
	f': X' := X \times_B B' \to B'.
\]
The smooth topology of the $\Sigma_g$-bundle $f': X' \to B'$ can be described as follows. Fix a closed disk $D^2 \subseteq B$ with interior $\mathring{D^2} \subseteq B$ and write
\[
	B_1 = B_2 = B - \mathring{D^2}.
\]
Then $B'$ can be identified smoothly with 
\[
	B' \cong B_1 \cup_{\partial} B_2,
\]
where $B_1$ and $B_2$ are glued along the boundary $\partial B_i \cong S^1$ by a reflection $r: S^1 \to S^1$. 

Similarly, let 
\[
	X_1 = X_2 = f^{-1}(B - \mathring{D^2}).
\]
Because $D^2 \subseteq B$ is contractible, the restrictions $f^{-1}(D^2) \to D^2$ and $f^{-1}(\partial D^2) \to \partial D^2$ are isomorphic to the trivial fiber bundles $\Sigma_g \times D^2 \to D^2$ and $\Sigma_g \times S^1 \to S^1$ respectively for both $i = 1, 2$. The total space $X'$ can then be identified smoothly with 
\[
	X' \cong X_1 \cup_{\partial} X_2,
\]
where $X_1$ and $X_2$ are glued along the boundary $\partial X_i \cong \Sigma_g \times S^1$ by the diffeomorphism $\id_{\Sigma_g} \times r: \Sigma_g \times S^1 \to \Sigma_g \times S^1$. In terms of these identifications, the map $f': X_1 \cup_{\partial} X_2 \to B_1 \cup_{\partial} B_2$ restricts to $f|_{X_i}: X_i \to B_i$ for each $i =1, 2$. 

By Van Kampen's theorem, there are isomorphisms of groups 
\begin{align*}
	\pi_1(B') &\cong \pi_1(B_1) *_{\pi_1(S^1)} \pi_1(B_2), \\
	\pi_1(X') &\cong \pi_1(X_1) *_{\pi_1(S^1 \times \Sigma_g)} \pi_1(X_2).
\end{align*}
Throughout this appendix, we identify $\pi_1(B_i)$ and $\pi_1(X_i)$ with their images in $\pi_1(B')$ and $\pi_1(X')$ respectively induced by the inclusions $B_i \hookrightarrow B'$ and $X_i \hookrightarrow X'$ for both $i = 1 ,2$. These identifications are well-defined because the maps $\pi_1(B_i) \to \pi_1(B')$ and $\pi_1(X_i) \to \pi_1(X')$ induced by these inclusions are injective for both $i = 1, 2$. 

Let $s_i: B_i \to X_i$ denote the restriction of $s: B \to X$ to $B_i$ for both $i = 1, 2$. 
\begin{definition}
For any $\gamma \in \pi_1(\Sigma_g)$, let $s_\gamma: \pi_1(B') \to \pi_1(X')$ be defined
\[
	s_\gamma(\ell) = \begin{cases}
	(s_1)_*(\ell) \in \pi_1(X_1) & \text{ if } \ell \in \pi_1(B_1) \\
	\gamma \cdot (s_2)_*(\ell) \cdot \gamma^{-1}\in \pi_1(X_2) & \text{ if } \ell \in \pi_1(B_2).
	\end{cases}
\]
\end{definition}
\begin{lemma}
For any $\gamma \in \pi_1(\Sigma_g)$, the map $s_\gamma$ is well-defined. Furthermore, $s_\gamma$ is a section of $(f')_*: \pi_1(X') \to \pi_1(B')$.
\end{lemma}
\begin{proof}
To check that $s_\gamma$ is well-defined, it suffices to show that
\[
	(s_1)_*(\ell) = \gamma \cdot (s_2)_*(\ell) \cdot \gamma^{-1}
\]
for $\ell$ a generator of $\pi_1(S^1) = \pi_1(\partial D^2)$. For both $i = 1, 2$, the image $s_i(\partial D^2)$ in $f^{-1}(\partial D^2) \cong \Sigma_g \times S^1$ bounds the disk $s(D^2)$ in $f^{-1}(D^2) \cong \Sigma_g \times D^2$, and so
\[
	(s_i)_*(\ell) = (1, \ell) \in \pi_1(\Sigma_g) \times \pi_1(S^1) \cong \pi_1(f^{-1}(\partial D^2)).
\]
Therefore, $(s_i)_*(\ell)$ commutes with $\gamma \in \pi_1(\Sigma_g) \leq \pi_1(f^{-1}(\partial D^2))$, and 
\[
	(s_1)_*(\ell) = (s_2)_*(\ell) = \gamma \cdot (s_2)_*(\ell) \cdot \gamma^{-1}.
\]

For any $\ell \in \pi_1(B_1)$ or for any $\ell \in \pi_1(B_2)$, compute that
\[
(f')_* \circ s_{\gamma}(\ell) = \begin{cases} f_* \circ (s_1)_*(\ell) \in \pi_1(B_1) & \text{ if }\ell \in \pi_1(B_1) \\
f_*(\gamma) \cdot f_*\circ (s_2)_*(\ell) \cdot f_*(\gamma^{-1})\in \pi_1(B_2) & \text{ if } \ell \in \pi_1(B_2).
\end{cases}
\]
Note that $f_* \circ (s_i)_* = \id_{\pi_1(B_i)}$ for both $i = 1, 2$, and that $f_*(\gamma) = 1 \in \pi_1(B_i)$ for any $\gamma \in \pi_1(\Sigma_g)$. In particular, $(f')_* \circ s_\gamma(\ell) = \ell$ for any $\ell \in \pi_1(B')$, and hence $s_\gamma$ is a section of $(f')_*$. 
\end{proof}

The following lemma provides a criterion to determine whether two group-theoretic sections $s_\gamma$ and $s_\delta$ differ by conjugation by an element of $\pi_1(\Sigma_g)$.
\begin{lemma}\label{lem:non-conj-criterion}
If $s_\gamma$ is conjugate to $s_\delta$ by an element of $\pi_1(\Sigma_g)$ then $[\delta] - [\gamma]$ is contained in $H_1(\Sigma_g, \Z)^{\pi_1(B)}$, the subgroup fixed by the monodromy action of the Kodaira fibration $f: X \to B$.
\end{lemma}
\begin{proof}
Suppose that there exists some $\beta \in \pi_1(\Sigma_g)$ such that 
\[
	s_\gamma(\ell) = \beta \cdot s_\delta(\ell) \cdot \beta^{-1} \in \pi_1(X')
\]
for all $\ell \in \pi_1(B')$. Because the maps $\pi_1(X_i) \to \pi_1(X')$ for $i =1, 2$ induced by inclusions $X_i \hookrightarrow X'$ are injective group homomorphisms, the following equalities also hold in $\pi_1(X_1)$ and $\pi_1(X_2)$: 
\begin{align}
	\label{eqn:s1} (s_1)_*(\ell) &= \beta \cdot (s_1)_*(\ell) \cdot \beta^{-1} \in \pi_1(X_1) & \text{ if } \ell \in \pi_1(B_1), \\
	\label{eqn:s2} \gamma \cdot (s_2)_*(\ell) \cdot \gamma^{-1} &= \beta \cdot (\delta \cdot (s_2)_* (\ell) \cdot \delta^{-1}) \cdot \beta^{-1} \in \pi_1(X_2) & \text{ if } \ell \in \pi_1(B_2).
\end{align}
Recall that under the identifications $X_1 = X_2$ and $B_1 = B_2$, the sections $s_1$ and $s_2$ agree. Therefore, (\ref{eqn:s2}) implies that 
\begin{align*}
	(s_1)_*(\ell) &= \gamma^{-1} \cdot \beta \cdot (\delta \cdot (s_1)_* (\ell) \cdot \delta^{-1}) \cdot \beta^{-1}\cdot \gamma \in \pi_1(X_1) & \text{ if } \ell \in \pi_1(B_1).
\end{align*}
Combining with (\ref{eqn:s1}), we obtain for all $\ell \in \pi_1(B_1)$
\[
	(s_1)_*(\ell) = (\beta^{-1} \gamma^{-1} \beta \delta) \cdot (s_1)_*(\ell) \cdot (\beta^{-1} \gamma^{-1} \beta \delta)^{-1} \in \pi_1(X_1),
\]
i.e. $(\beta^{-1} \gamma^{-1} \beta \delta) \in \pi_1(\Sigma_g)$ commutes with $(s_1)_*(\ell) \in \pi_1(X_1)$ for all $\ell \in \pi_1(B_1)$. Because the map $\pi_1(B_1) \to \pi_1(B)$ induced by the inclusion $B_1 \hookrightarrow B$ is surjective and because $s_1: B_1 \to X_1$ is the restriction of $s: B \to X$, it also holds that $(\beta^{-1} \gamma^{-1} \beta \delta) \in \pi_1(\Sigma_g)$ commutes with $s_*(\ell) \in \pi_1(X)$ for all $\ell \in \pi_1(B)$.

Let $\rho: \pi_1(B) \to \Aut(H_1(\Sigma_g, \Z))$ denote the monodromy representation of $f: X \to B$. Then $\rho$ can be described using the section $s: B \to X$ in the following way. For any $\alpha \in \pi_1(\Sigma_g)$ and any $\ell \in \pi_1(B)$, the action of $\ell$ on the homology class $[\alpha] \in H_1(\Sigma_g, \Z)$ is
\[
	\rho(\ell)([\alpha]) = [\, \underbrace{s_*(\ell)^{-1} \cdot \alpha \cdot s_*(\ell)}_{\in \pi_1(\Sigma_g) \trianglelefteq \pi_1(X)} \, ] \in H_1(\Sigma_g, \Z).
\]
Therefore, for any $\ell \in \pi_1(B)$,
\[
	\rho(\ell)([\beta^{-1} \gamma^{-1} \beta \delta]) = [\beta^{-1} \gamma^{-1} \beta \delta]
\]
because $s_*(\ell)$ commutes with $\beta^{-1} \gamma^{-1} \beta \delta$. In other words, $[\beta^{-1} \delta^{-1}\beta \gamma] \in H_1(\Sigma_g, \Z)$ is fixed under the action of $\pi_1(B)$, i.e. $[\beta^{-1} \gamma^{-1} \beta \delta] = -[\gamma] + [\delta]$ is contained in $H_1(\Sigma_g, \Z)^{\pi_1(B)}$.
\end{proof}

We apply this criterion to finish the proof of Proposition \ref{Prop: surjectivity of top. section conjecture} part (2).
\begin{proof}[Proof of Proposition \ref{Prop: surjectivity of top. section conjecture} part (2)]
For any $\gamma, \delta \in \pi_1(\Sigma_g)$, the associated sections $s_\gamma$ and $s_\delta$ are not conjugate by an
element of $\pi_1(\Sigma_g)$ if $[\gamma] \neq [\delta]$ as elements of the quotient $H_1(\Sigma_g, \Z)/H_1(\Sigma_g,
\Z)^{\pi_1(B)}$ by Lemma \ref{lem:non-conj-criterion}. By (a dual version of) Remark \ref{Remark: monodromy of Kodaira fibration}, the monodromy action $\rho: \pi_1(B) \to \Aut(H_1(\Sigma_g, \Z))$ of $f: X \to B$ is nontrivial. In other words, the quotient 
\[
	H_1(\Sigma_g, \Z)/H_1(\Sigma_g, \Z)^{\pi_1(B)}
\]
is an infinite free abelian group. Therefore, there exist infinitely many pairwise nonconjugate sections of the short exact sequence (\ref{top. SES}) of the form $s_\gamma$ with $\gamma \in \pi_1(\Sigma_g)$, and the map $\textrm{sec}_{\textrm{top}}$ has infinite codomain. On the other hand, the geometric Mordell conjecture for Kodaira fibrations \cite{Man63} says that any Kodaira fibration has finitely many algebraic sections, i.e. the map $\textrm{sec}_{\textrm{top}}$ has finite domain. Therefore, $\textrm{sec}_{\textrm{top}}$ is not surjective.
\end{proof}

\bibliographystyle{alpha}
\bibliography{bibliography-algebraicsection.bib}

\end{document}